\documentclass[11pt]{article}

\usepackage{amsmath}
\usepackage{amsthm}
\usepackage{amsfonts}
\usepackage{setspace}
\usepackage{fullpage}
\usepackage{amssymb}
\usepackage{enumitem}
\usepackage{bbold} 
\usepackage{mathtools}
\usepackage{comment}

\usepackage[utf8]{inputenc}

\usepackage{tikz}
\usetikzlibrary{patterns,snakes} 
\usepackage{float}
\usepackage{caption}
\usepackage{graphicx}
\newtheorem{lemma}{Lemma}
\newtheorem{theorem}{Theorem} 
 \newtheorem{proposition}{Proposition}

\newcommand{\cA}{\mathcal{A}}
\newcommand{\cB}{\mathcal{B}}
\newcommand{\cC}{\mathcal{C}}

\newcommand{\F}{\mathcal{F}}
\newcommand{\E}{\mathcal{E}}
\newcommand{\cE}{\mathcal{E}}

\newcommand{\cH}{\mathcal{H}}
\newcommand{\cG}{\mathcal{G}}
\newcommand{\ex}{{\rm ex}}
\newcommand{\ind}{{\rm\shortminus{ind}}}

\newcommand{\NN}{\mathbb{N}}

\usepackage{amsfonts} 
\DeclareMathSymbol{\shortminus}{\mathbin}{AMSa}{"39}

\newcommand{\blue}{\text{blue}}

\date{}
\newcounter{num}

\title{Induced Tur\'an problem in bipartite graphs}

\author{Maria Axenovich\thanks{Karlsruhe Institute of Technology, Karlsruhe, Germany, \texttt{maria.aksenovich@kit.edu}.} \and Jakob Zimmermann\thanks{Karlsruhe Institute of Technology, Karlsruhe, Germany}}

\begin{document}
\maketitle

\abstract{The classical extremal function for a graph $H$, $\ex(K_n, H)$ is the largest number of edges in a subgraph of $K_n$ that contains no subgraph isomorphic to $H$. 
Note that defining $\ex(K_n, H\ind)$ by forbidding induced subgraphs isomorphic to $H$ is not very meaningful for a non-complete $H$ since one can avoid it by considering a clique.

For graphs $F$ and $H$, let $\ex(K_n, \{F, H\ind\})$ be the largest number of edges in an $n$-vertex graph that contains no subgraph isomorphic to $F$ and no induced subgraph isomorphic to $H$. 
Determining this function asymptotically reduces to finding either $\ex(K_n, F)$ or $\ex(K_n, H)$, unless $H$ is a biclique or both $F$ and $H$ are bipartite. Here, we consider the bipartite setting, $\ex(K_{n,n}, \{F, H\ind\})$ when $K_n$ is replaced with $K_{n,n}$, $F$ is a biclique, and $H$ is a bipartite graph.

Our main result, a strengthening of a result by Sudakov and Tomon, implies that for any $d\geq 2$ and any $K_{d,d}$-free bipartite graph $H$ with  each vertex in one part of degree either at most $d$ or a full degree, so that there are at most $d-2$ full degree vertices in that part,  
one has $\ex(K_{n,n}, \{K_{t,t}, H\ind\}) = o(n^{2-1/d})$.
This provides an upper bound on the induced Tur\'an number for a wide class of bipartite graphs and implies in particular an extremal result for bipartite graphs of bounded VC-dimension by Janzer and Pohoata.
}

\section{Introduction}\label{sec:intro}
For  graphs $G$ and  $H$, we say that $G$ contains a {\it copy} (an {\it induced copy}) of $H$ if $G$ has a subgraph (induced subgraph) isomorphic to $H$. We also say that $G$ contains $H$ as a subgraph (as an induced subgraph) in respective cases, and write $H\subseteq G$ ($H\subseteq _{\rm ind} G$). If $G$ doesn't contain a copy of $H$, we say that $G$ is $H$-{\it free}.
The classical extremal function $\ex(\cG, H)$ for a fixed graph $H$ and a ground graph $\cG$ is defined as the largest number of edges in a subgraph of $\cG$ that contains no copy of $H$. This function is a subject of extensive studies over the last several decades, in particular for $\cG=K_n$, a complete graph on $n$ vertices or $\cG= K_{n,n}$, a complete bipartite graph with $n$ vertices in each part. 
Note that that defining $\ex(K_n, H\ind)$ by forbidding $H$ as an induced subgraph is not very meaningful since for non-complete $H$, one can avoid it as an induced subgraph by considering a clique. However, when one forbids a graph $H$ as an induced subgraph and another graph just as a subgraph, the analogous question becomes non-trivial. \\

Here, for a graph $F$ on at least two vertices and a non-empty graph $H$, let $\ex(\cG, \{F, H\ind\})$ be the largest number of edges in a subgraph of $\cG$ that contains no subgraph isomorphic to $F$ and no induced subgraph isomorphic to $H$. Note that for an empty $H$, $\ex(K_n, \{F, H\ind\})$ is not well-defined for large $n$ as a consequence of Ramsey's theorem.\\

Here, we are concerned with the case when $F$ is complete bipartite and $H$ is from a large class of bipartite graphs.
We shall formulate our result in terms of a specific graph $W(k, d, r)$ that contains, as an induced subgraph, any bipartite $K_{d,d}$-free graph with each vertex in one part of degree at most  $d$ or a full degree, for some sufficiently large $k$ and some $r$ (see Lemma \ref{lem:propertiesW} for a formal proof).\\

First we define a hedgehog graph.  If  $k,d, $ and $j$ are positive integers, where $k\geq d$, we say that a bipartite graph is a $(k, d, j)$-{\it hedgehog} and denote it $H(k, d,j)$  if it has a part $V$ of size $k$, called the {\it body of the hedgehog},  and a part $U$ that is a pairwise disjoint union of sets $U_E$,  each of size $j$,  such that for each $E\in \binom{V}{d}$ each vertex of $U_E$ is adjacent to each vertex of $E$, and there are no edges between $U_E$ and any vertex not in $E$. For integers $r, k, d$, where $0\leq r \leq d-2 \leq k-2$, we define a bipartite graph $W=W(k, d, r)$ as a union of a hedgehog $H(k, d, d-r-1)$ and a complete bipartite graph with parts $V$ and $Y$, where $V$ the the body of the hedgehog and $Y$ is disjoint from the vertex set of the hedgehog and has size $r$. We call the part $V$ {\it a body} of $W$. See Figure  \ref{fig:W}.\\

\begin{figure}[h!]
\begin{center}
\includegraphics[width=0.6\textwidth]{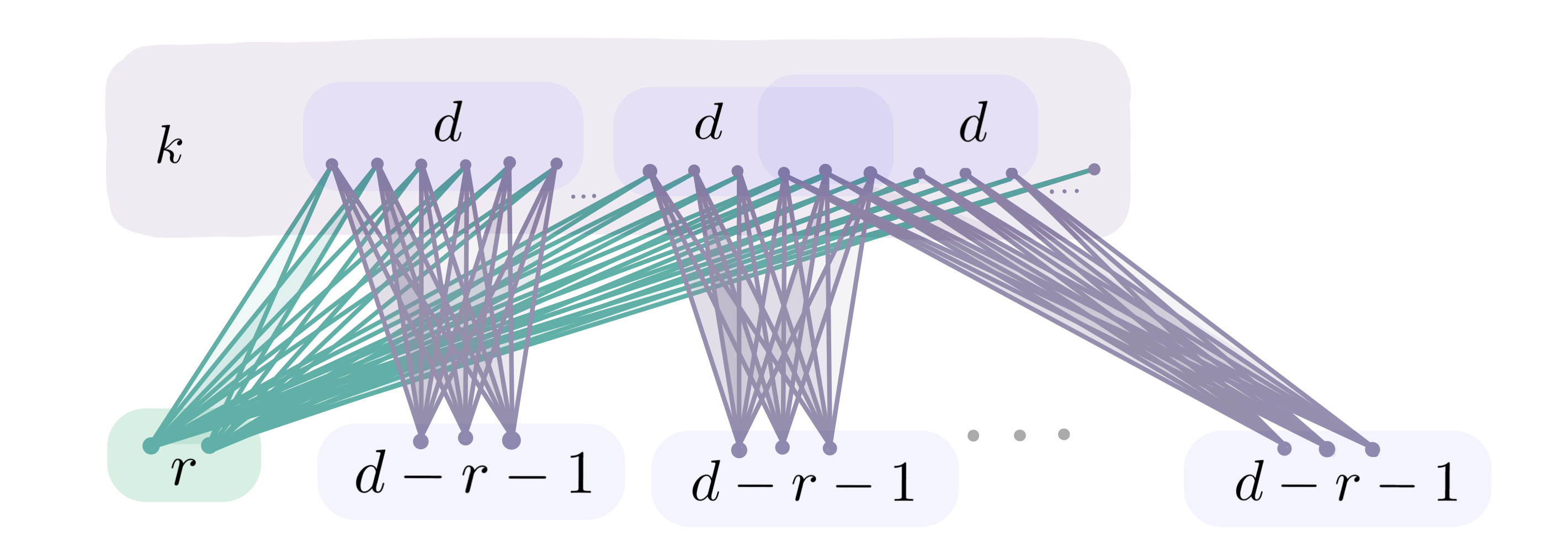}
\end{center}
\caption{A graph $W(k, d, r)$ with $r=2$ and $d=6$.}
\label{fig:W}
\end{figure}
 
\begin{theorem}\label{thm:main}
 Let $k, r, d, t $ be fixed non-negative integers, $r+2 \leq d \leq k$, and $\epsilon$ be an arbitrary positive number. Then there is $n_0$ such that for any $n>n_0$, 
 if $G=(A, B; E)$ is a bipartite graph, $|A|=|B|=n$ and $||G|| \geq \epsilon n^{2-\frac{1}{d}}$, then $G$ contains either a copy of $K_{t,t}$ or an induced copy of $W(k,d,r)$ with body in $A$.
 In particular, 
 $\ex(K_{n,n}, \{K_{t,t}, W(k,d,r)\ind\})= o(n^{2-\frac{1}{d}})$.
 Moreover, for any $\delta>0$,  $t\geq 2d-1$ and sufficiently large $k$, $\ex(K_{n,n}, \{K_{t,t}, W(k,d,r)\ind\})= \Omega(n^{2-\frac{1}{d}-\delta}) $. 
\end{theorem}

If $r=0$, we have that $W(k, 2, 0)= H(k, 2, 1)$, that is isomorphic to a $1$-subdivision of $K_k$, denoted $K_k'$.
Note that in this case, for any $k\geq 3$, $\ex(K_{n,n}, K_k')=O(n^{3/2-1/(4k-6)})$, as shown by Janzer \cite{J}, see also an earlier result by Conlon and Lee \cite{CL}. So, in this case Theorem \ref{thm:main} gives a weaker statement. \\

Our proof of Theorem \ref{thm:main} in Section \ref{sec:main-proofs} merges the ideas of Sudakov and Tomon as well as Janzer and Pohoata who proved the following two theorems. We shall show that these theorems follow from Theorem  \ref{thm:main}.

\begin{theorem}[Sudakov and Tomon \cite{sudakov2019turan}] \label{thm:ST}
 Let $d \geq 2$ be an integer and $H$ be a $K_{d,d}$-free bipartite graph $H$ 
 such that every vertex in one part of $H$ has degree at most $d$. Then $\ex(K_n,H) = o(n^ {2 - \frac{1}{d}})$.
\end{theorem}
A related result by F\"uredi \cite{F91} and Alon, Krivelevich, and Sudakov \cite{AKS} shows that $\ex(n, H) = O(n^ {2 - \frac{1}{d}})$, for a bipartite graph $H$ such that every vertex in each part has degree at most $d$, that is tight since $\ex(n, K_{s,d}) = \Omega(n^{2-1/d})$, for sufficiently large $s$.
The following result is a strengthening of a result by Fox et al. \cite{FPSSZ}. We formally define the VC-dimension in Section \ref{sec:lemmas}.

\begin{theorem}[Janzer and Pohoata \cite{janzer2021zarankiewicz}] \label{thm:JP}
 Let $k$ and $d$ be integers such that $t\geq d\geq 3$. Let $G$ be a bipartite graph on partite sets of sizes $n/2$ each and with VC-dimension at most $d$ with respect to one of its sides. If $K_{t,t}\nsubseteq G$ then $|E(G)| = o(n^{2-\frac{1}{d}})$.
\end{theorem} 
Note that Theorem \ref{thm:JP} doesn't hold for $d=2$ as can be seen by considering a balanced $K_{2,2}$-free bipartite graph on $2n$ vertices
and $\Omega(n^{3/2})$ edges, see for example \cite{F}. \\
Note also that Theorem \ref{thm:main} is strictly stronger than  Theorem \ref{thm:ST} because it about forbidden graphs with some vertices of high degree. In addition 
Theorem \ref{thm:main} is strictly stronger than  Theorem \ref{thm:JP} because having the VC-dimension greater than $d$ implies the existence of a graph that is a proper induced subgraph of  $W(k,d,r)$.\\

The question about extremal properties of graphs with forbidden induced subgraphs is not new and was studied a lot in the framework of Erd\H{o}s-Hajnal conjecture \cite{EH}, asserting that any graph $G$ without a given induced subgraph has a large (of size $|G|^\beta$, for some positive constant $\beta$) clique or a co-clique. While in general the conjecture is open, for bipartite graphs it is shown to be true by Erd\H{os}, Hajnal, and Pach \cite{EHP}, see also \cite{ATW} for a more quantitative version. For graphs with bounded VC-dimension, it was shown that the conjecture holds by  Nguyen,  Scott,  and Seymour \cite{NSS},  see also a paper by  Fox, Pach, and Suk \cite{FPS}. In addition, the so-called Ramsey-Tur\'an number ${\rm RT}(n,m,F)$, introduced by S\'os, is also connected to the above question, see \cite{SS} and \cite{BHS}. Here, ${\rm RT}(n, m , F)$ is the largest number of edges in an $n$-vertex graph that doesn't contain $F$ as a subgraph and doesn't contain an empty graph on $m$ vertices as an induced subgraph. However, here, $m$ is typically considered to be a function of $n$.\\

Finding $\ex(K_n, \{F, H\ind\})$ reduces to determining $\ex(K_n, H)$ or $\ex(K_n, F)$ asymptotically, unless both $F$ and $H$ are bipartite, or $H$ is complete bipartite and $F$ is not bipartite, see Lemma \ref{lem:ind}. The case when $H$ is complete bipartite was addressed by Loh et al. \cite{LTTZ}, Illingworth \cite{I}, as well as Ergemlidze, Gy\H{o}ri, and Methuku \cite{EGM}. When $F=K_{t,t}$ and $H$ is a forest, a result by Scott, Seymour, and Spirkl \cite{SSS}, within a $\chi$-boundedness project, implies that $\ex(K_n, \{F, H\ind\}) \leq t^cn$, for some constant $c=c(H)$. See also an earlier result by Bonamy et al. \cite{BBPRTW} when $H$ is a path.  For bipartite $H$, 
Lemma 7.1. in  Du, Gir\~ao, Hunter, McCarty,  and Scott \cite{GH} and  Theorem 1.4. in  Bourneuf,  Buci\'c, Cook, and Davies  \cite{BBCD}, obtained independently and using different approaches,  claim that $\ex(K_n, \{K_{t,t}, H\ind\}) \leq O(n^{2- \epsilon(H)})$, where the best value for $\epsilon(H)$ is  $1/(100\Delta(H))$ given in \cite{GH}.  Note that  Hunter, Milojevi\'c, Sudakov, and Tomon \cite{HMST} showed the following bound: $\ex(K_n, \{K_{t,t}, H\ind\}) \leq (4|A||B|t)^{4(|A|+|B|)+10} n^{2-1/d}$, for a bipartite graph $H$ with parts $A$ and $B$ such that each vertex in $B$ has degree at most $d$.  \\

The rest of the paper is structured as follows. 
We list some standard tools and definitions, as well as  Lemma \ref{lem:propertiesW} on properties of the graph $W(k,d,r)$ and Lemma \ref{lem:ind}  on properties of $\ex(n, \{F, H\ind\})$ in Section \ref{sec:lemmas}. We  prove Theorem \ref{thm:main} and derive Theorems  \ref{thm:ST} and \ref{thm:JP} from it in Section \ref{sec:main-proofs}.

\section{ Notations, Definitions, and Preliminary Lemmas}\label{sec:lemmas}

For a hypergraph $\F = (V, \E)$ with vertex set $V$ and edge set $\E \subseteq 2^V$, we often identify $\F$ with the set of edges $\E$. We denote the number of vertices in $\F$ by $|\F|$ and the number of edges in $\F$ by $||\F||$.
A hypergraph is $d$-{\it uniform} if each hyperedge has size $d$.

We shall denote a bipartite graph $G$ with partite sets (parts) $A$ and $B$ and edge set $E$ by $G=(A, B; E)$.
For bipartite graphs $G=(A, B; E)$ and $G'=(A',B'; E')$, we say that $G'$ is a subgraph of $G$ {\it respecting sides} if $G'$ is a subgraph of $G$, $A'\subseteq A$ and $B'\subseteq B$. \\

For a hypergraph $\F=(V, \E)$, or respective set system $\E$, the {\it VC-dimension} of $\F$
 is the largest integer $d$ such that there is a subset $S$ of $V$ that is shattered by edges of $\F$, i.e., such that for any subset $S'$ of $S$ there is $E\in \E$ so that $S\cap E=S'$. This notion was introduced by Vapnik and Chervonenkis \cite{VC}.
The {\it VC-dimension of a bipartite graph} $G=(A\cup B, E)$ with parts $A$ and $B$, {\it with respect to $A$}, denoted 
 is the VC-dimension of the hypergraph with vertex set $A$ and edge set corresponding to the neighbourhoods of vertices from $B$.
The {\it incidence graph} of the hypergraph $\F$ is a bipartite graph with parts $V$ and $\E$ and set of edges $\{\{v, E\}: v\in V, v\in E, E\in \E\}$. Note that a bipartite graph $G$ with parts $A$ and $B$ has VC-dimension at most $d$ with respect to $A$ 
if and only if $G$ does not contain an induced copy of the incidence graph of the hypergraph
with vertex set $[d+1]$ and edge set $2^{[d+1]}$, such that the copy of $[d+1]$ is in $A$. Note that $H(k, d,1)$ corresponds to the incidence graph of the complete $d$-uniform hypergraph on $k$ vertices.\\ 
 
For a graph $G$ and disjoint vertex sets $A$ and $B$, we write $A\sim B$ if each vertex in $A$ is adjacent to each vertex in $B$. We further write $G[A,B]$ to denote a bipartite subgraph of $G$ with parts $A$ and $B$ and containing all edges with one endpoint in $A$ and another in $B$; by $||A,B||$ we denote the number of edges in $G[A,B]$, when $G$ is clear from context. By $N(A)=N_G(A)$, the neighbourhood of $A$, we denote the maximal set $B$ of vertices such that $A\sim B$. 
Note that this definition differs from the often used notion of a set neighbourhood defined as the union of it vertex neighbourhoods. Here, $N(A)$ is the intersection of neighbourhoods of vertices from $A$.
\\

For a graph $G$, we denote by $\Delta(G)$ and $\delta(G)$ the maximum and minimum degree of $G$, respectively. 
For $K>0$, a graph $G$ is $K$-{\it almost regular} if $\Delta(G)/\delta(G) \leq K$. 
For an $\ell$-uniform hypergraph $\cH$, a $q$-{\it clique} in $\cH$ is a set $S$ of $q$ vertices such that the edge set of $\cH$ contains all $\ell$-element subsets of $S$.\\

Most items of the  following Lemma are given by Illingworth \cite{I-1}. We include the proof of a slightly more precise statement for completeness.
\begin{lemma}\label{lem:ind} Let $F$ be a graph, $H$ be a non-empty graph, and $r = \min \{\chi(H), \chi(F) \}$. 
\begin{enumerate}
\item{} Let $r\geq 3$. 
If $H$ is a complete multipartite graph, then $\ex(n, \{F, H\ind\}) = (1-1/(r-1))\binom{n}{2}(1+o(1))$.
Otherwise $\ex(n, \{F, H\ind\}) = (1-1/(\chi(F)-1))\binom{n}{2}(1+o(1)).$ In either case $\ex(n, \{F, H\ind\})\in \{ \ex(n, H)(1+o(1)), \ex(n, F)(1+o(1))\}$.
\item{} Let $r=2$ and $\chi(H)>2$. Then 
$\ex(n, F)/2 \leq \ex(n, \{F, H\ind\}) \leq \ex(n, F)$. 
\item{} Let $r=2$, $\chi(F)>2$, and $H$ be not complete bipartite. Then 
$\ex(n, \{F, H\ind\}) = \ex(n, F)(1+o(1))$. 
\end{enumerate}
\end{lemma}

\begin{proof}
\begin{enumerate}
\item{} We have that $F, H\not\subseteq T_{r-1}(n)$, where $T_{r-1}(n)$ is the Tur\'an graph with $r-1$ parts. Thus $\ex(n, \{F, H\ind \})\geq ||T_{r-1}(n)|| = (1-1/(r-1))\binom{n}{2}(1+o(1))$. 
If $H$ is not a complete multipartite graph, then $H$ is not an induced subgraph of any complete multipartite graph. In particular, $H$ is not an induced subgraph of $T_{\chi(F)-1}(n)$. Since $F\not\subseteq T_{\chi(F)-1}(n)$, we have $\ex(n, \{F, H\ind \})\geq ||T_{\chi(F)-1}(n)|| = (1-1/(\chi(F)-1))\binom{n}{2}(1+o(1))$.\\
\sloppy
Since $\ex(n, \{F, H\ind\})\leq \ex(n, F)$, using Erd\H{o}s-Stone Theorem
we have that 
$\ex(n, \{F, H\ind\})\leq (1-1/(\chi(F)-1))\binom{n}{2}(1+o(1))$.
If $r= \chi(H) < \chi(F)$ and $H$ is complete $r$-partite, we see that $\ex(n, \{F, H\ind\})\leq \ex(n, T_r(t))$, for sufficiently large $t$. Indeed, if $T=T_r(t)\subseteq G$, then either some part of $T$ induces a subgraph of $G$ that contains $F$ as a subgraph or each part of $T$ induces a large independent set by Ramsey Theorem. In the latter case $G$ contains an induced copy of $H$. So, $\ex(n, \{F, H\ind\})\leq (1-1/(r-1))\binom{n}{2}(1+o(1))$.\\

\item{} Let $F$ be bipartite and $H$ not bipartite. Then any $F$-free bipartite graph doesn't contain $H$ as a subgraph, so by taking a densest bipartite subgraph of an $F$-free $n$-vertex graph, we have $\ex(n, \{F, H\ind\}) \geq \ex(n, F)/2$. In addition, $\ex(n, \{F, H\ind\}) \leq \ex(n, F).$ 

\item{}
If $H$ is bipartite, but not complete bipartite, and $\chi(F)\geq 3$, we have that $T_{r-1}(n)$ contains neither $H$ as an induced subgraph, nor $F$ as a subgraph, so $\ex(n, \{F, H\ind\}) \geq ||T_{r-1}(n)|| =\ex(n, F) (1+o(1))$.
Since $\ex(n, \{F, H\ind\}) \leq \ex(n, F)$, the result follows.
\end{enumerate}
\end{proof}

The only remaining cases are when $H$ is a biclique or when both $H$ and $F$ are bipartite.
Regarding the former case Loh et al. \cite{LTTZ} proved that $\ex(n, \{K_r, K_{s,t}\ind\}) = O( n ^ { 2 - \frac{ 1 }{ s } } )$
asymptotically matching the upper bound on $ex(n, K_{s,t})$ by K\H{o}v\'ari-S\'os-Tur\'an in Lemma \ref{thm:KST}.\\

We shall use the following result by Nagl, R\"odl, and Schacht, \cite{NRG}, also proved by Gowers \cite{G}. 
\begin{lemma}[Hypergraph Removal Lemma]
 For any $\epsilon>0$, any integers $q$ and $d$, $q\geq d\geq 2$, there is $\delta= \delta(\epsilon, d, q)>0$ such that if $\cH$ is a $d$-uniform hypergraph and one needs to delete at least $\epsilon \binom{|\cH|}{d}$ hyperedges to destroy all $q$-cliques, then number of $q$-cliques in $\cH$ is at least $\delta \binom{|\cH|}{q}$.
\end{lemma}

\begin{lemma}[K\H{o}vari-S\'os-Tur\'an Theorem, \cite{KST}] \label{thm:KST}
If $G=(Y_1, Y_2; E)$ is a bipartite graph, such that $G$ doesn't contain a complete bipartite subgraph with $y_1$ vertices in $Y_1$ and $y_2$ vertices in $Y_2$, then 
$||G||\leq (y_2-1)^{1/y_1}(|Y_1|-y_1+1)|Y_2|^{1-1/y_1}+(y_1-1)|Y_2|$.
\end{lemma}

 The following lemma appears as a Proposition 2.7 in Jiang and Seiver paper \cite{JS}, see a similar earlier result by Erd\H{o}s and Simonovits \cite{ES}. The authors do not state that the resulting subgraph is induced. However, the proof of their proposition implies that the subgraph is induced since it is obtained by an iterative vertex deletion procedure.

\begin{lemma} [Reduction Lemma] \label{lem:reduction}
 For every $\epsilon \in (0, 1)$ there is $K > 0$ such that for every $c \geq 1$ there is an integer $N \in \NN$ such that for every integer $n \geq N$ and every $n$-vertex graph $G$ with $||G|| \geq c n ^ { 1 + \epsilon }$
 there is an induced subgraph $G' \subseteq G$ on $m \geq n ^ { \frac{ \epsilon }{ 2 } \frac{ 1 - \epsilon }{ 1 + \epsilon } }$ vertices such that $||G'||\geq \frac{2c}{5}m^{1+\epsilon}$ and $G'$ is $K$-almost regular, for $K= K(\epsilon)$.
\end{lemma}

The following lemma is a classical application of the probablilistic deletion method that is mentioned in particular in Theorem 2.26 of a survey by F\"uredi and Simonovits \cite{FS}.
\begin{lemma}[Deletion-method Lemma]\label{lem:deletion}
Let $\cH$ be a family of non-empty graphs, then there are positive constants $c$ and $n_0$ such that for any $n>n_0$ there is a bipartite graph $G$ with $n$ vertices in each part, not containing any member of $\cH$ as a subgraph and such that $||G||\geq cn^{2-\gamma}$, where $\gamma(\cH) =\max \{\gamma(H): H\in \cH\}$ and $\gamma(H)= \frac{|H|-2}{||H||-1}$.
\end{lemma}

\noindent
\begin{lemma}\label{lem:propertiesW}
 Let $d$, $r$, and $n$ be non-negative integers, $d\geq r+2$. Then there is an integer $k$ such that the following holds. 
 \begin{enumerate}
 \item{}\label{1}
 If $G'=(V', U'; E)$ is a bipartite graph such that each vertex from $U'$ has degree at most $d$ and there is no copy of a complete bipartite graph with $d-r$ vertices in $U'$ and $d$ vertices in $V'$, then for sufficiently large $k$, $H=H(k, d, d-r-1)$ and thus $W=W(k, d, r)$ contain an induced copy of $G'$ with set of vertices corresponding to $V'$ contained in the body of $H$ and in the body of $W$, respectively.\\
 \item{}\label{2}
 If $d\geq 3$ and $G$ is an incidence graph of the hypergraph with vertex set $[d+1]$ and edge set $2^{[d+1]}$, then $W=W(k, d, d-2)$ contains an induced copy of $G$ with the copy of $[d+1]$ contained in the body of $W$.\\
 \item{}\label{3}
 Let $s, k, d \in \mathbb{N}$ and $G=(V, U; E)$ be a bipartite graph with $|V|=k$, such that for each $d$-element subset $X$ of $V$ and each $y\in V\setminus X: \ |N(y)\cap N(X)|< (|N(X)|-s)/(k-d)$.
 Then $H(k, d, s)$ is an induced subgraph of $G$ with body $V$. 
 \end{enumerate}
\end{lemma}
 
\begin{proof}
$\ $
\begin{enumerate}
\item{}
Assume that $G'$ is a union of complete bipartite graphs with parts $A, U_A$, for all $A$ that are subsets of $V'$ of size at most $d$ and $U_A$'s are pairwise disjoint, $|U_A| = d-r-1$ if $|A|=d$ and $|U_A|=q$ for some large, but fixed $q$, say $q=(d-r-1)s$, for some $s$. We shall find an injective map of the vertex set of $G'$ into the vertex set of $H=H(k, d, d-r-1)$.
Let $V$ be the body of $H(k, d, d-r-1)$, assume that $V'\subseteq V$. 
 
For each $A\in \binom{V'}{d}$, we embed $U_A$ into the neighbourhood of $A$ in $H$. 
Consider the family $\cA$ of sets $A\subseteq V'$, $|A|<d$.
Let us choose for each $A\in \cA$ a set $\cB_A$ of $s$ pairwise disjoint sets from $V\setminus V'$, each of size $d$. This is possible if we choose $k$ large enough.
For each $A\in \cA$, let $\cB'_A= \{ B': B'\subseteq B, B\in \cB_A, |B'| = d-|A|\}$. 
For each member of $B'\in \cB'_A$ there is a set $U_{B'}$ of size $d-r-1$ in $H$ such that $U_{B'} \sim (B'\cup A)$ in $H$. Moreover, the $U_{B'}$'s are pairwise disjoint.
Thus their union is a set of $s(d-r-1)\binom{ d }{ |A| } \geq q$ vertices, each adjacent to every vertex of $A$. We embed the neighbourhood of $A$ in $G'$ into $\cup_{B'\in \cB'_A} U_{B'}$. 

\item{}
Note that $G$ has parts $[d+1]$ and $X= 2^{[d+1]}$. Let $G'$ be obtained by deleting the vertex $x$ of degree $d+1$ from $X$. Then $G'$ satisfies the conditions of the first part of the lemma with $r=d-2$ and $r\geq 1$ (here we use the fact that $d\geq 3$).
 So, $G'$ is an induced subgraph of $H=H(k, d, 1)$. We have that $W=W(k, d, d-2)$ is a supergraph of $H$ with a vertex, say $v$, not in $V(H)$ and adjacent to all vertices of the body of $H$. Then we embed $x$ in $v$ to obtain the copy of $G$ in $W$. 

\item{}
We greedily find sets $U_X\subseteq U$, for each $X\in \binom{V}{d}$, such that they are pairwise disjoint, of size $s$ each, such that $U_X\sim X$ and 
 $U_X \not\sim x$ for any $x\in V\setminus X$.
 For a fixed $X\in \binom{V}{d}$, we can ignore all neighborhoods $N(y)$, $y\in V\setminus X$, from $N(X)$. Then we have at least $|N(X)| - (k-d) (|N(X)|-s)/(k-d) \geq s$ vertices such that each of them is adjacent to all of $X$ and to none of vertices from $V\setminus X$. Choose a subset of size exactly $s$ from this set arbitrarily and call it $U_X$.
 We see that if $X\neq X'$, $X, X'\in \binom{V}{d}$, then $U_X\cap U_{X'}= \emptyset$, otherwise if $z\in U_X\cap U_{X'}$, by construction of $U_X$ and $U_{X'}$, $z\not\sim x, z\not\sim x'$, for $x\in X\setminus X'$ and $x' \in X'\setminus X$, but $z\sim X$ and $z\sim X'$, a contradiction. 
 In addition, by construction, for any $x'\in X'\setminus X$, $x'$ is not adjacent to any vertex from $U_{X}$. 
 Thus, $V$ and the sets $U_X$, $X\in \binom{V}{d}$ form a vertex set of an induced copy of $H(k, d, s)$ in $G$, with body $V$. 
 \end{enumerate}
\end{proof}

 \section{Proofs of the main results}\label{sec:main-proofs}
 
 The following result is a main tool in proving Theorem \ref{thm:main}.

\begin{proposition}\label{prop}
Let $k, d,t,s$ be positive integers, $d\leq k$. For every $ \eta> 0$ there are $ \xi, \kappa, n_0 > 0$ such that for any $K_{t,t}$-free bipartite graph $G=(A, B; E)$,
where $|B|\geq n_0$, $|A| \leq \xi |B| ^ { \frac{ s }{ d } }$, $||G|| \geq \xi \eta |B| ^ { \frac { d + s - 1 }{ d } }$, and any vertex in $B$ has degree at most $\kappa |A|$, 
there is an induced copy of the hedgehog $H(k, d, s)$ with body in $A$.
\end{proposition}

\begin{proof}
We shall set the constants as follows. Let $q $ be a sufficiently large constant, $q\gg k$.
We shall use the hypergraph Ramsey number $R= R_{d}(q)$, that is the smallest integer $N$ so that any coloring of $d$-element subsets of $[N]$ in two colors results in a monochromatic $q$-clique.\\

Consider the graph $G$ given in the statement of the proposition, let $|B|=n$.  Let $\cH=(A, \cE)$ be a hypergraph, where $\cE=\{A'\subseteq A: ~ |A'|=d, |N_G(A')|\geq s\}$.
Color the edges of $\cH$ red and blue according to the following rule: if $|N(A')|\geq q$, color $A'$ red, otherwise color it blue.
 We say that a $q$-clique in $\cH$ is blue if all edges of $\cH$ contained in it are blue. Similarly, a $q$-clique is red if all edges of $\cH$ contained in it are red.    We shall show that we can assume every $q$-clique to be blue. After that we count blue $q$-cliques in two different ways and arrive at a contradiction.\\

 {\bf Claim 1.~} If there is a red $q$-clique in $\cH$ then $H(k, d, s)$ is an induced subgraph of $G$ with body in $A$.\\

 \noindent
 Let $Q$ be a red $q$-clique.
 Let $X$ be a $d$-element subset of $Q$. We say that $y\in Q\setminus X$ {\it bad} for $X$ if $|N(y)\cap N(X)|\geq (|N(X)|-s)/(k-d)$. Otherwise, $y$ is {\it good } for $X$.\\

 Consider a $k$-element subset $V$ of $Q$. If for each $d$-element subset $X$ of $V$ and each $y\in V\setminus X$, $y$ is good for $X$, then by Lemma \ref{lem:propertiesW} (3)
 we have that $H(k, d, s)$ is an induced subgraph of $G$ with body $V$, and we are done. \\
 
 So, we can assume that for each $k$-element subset $V$ of $Q$ there is a $d$-element subset $X$ of $V$ and some $y\in V\setminus X$ that is bad for $X$. \\

 Let $X$ be a $d$-element subset of $V$. Let $X'$ be the set of all bad vertices for $X$ from $Q$. 
 Consider $G[X',N(X)]$. Since $G$ is $K_{t, t}$-free, the K\H{o}v\'ari-S\'os-T\'uran bound gives that 
 $||X', N(X)|| \leq (t-1)^{1/t} (|N(X)|-t+1)|X'|^{1-1/t} + (t-1)|X'| \leq c(t) |N(X)| |X'|^{1-1/t}$, for a constant $c(t)$. 
 Thus, $$|X'| \leq \frac{ ||X', N(X)||}{ (|N(X)|-s)/(k-d)} \leq \frac{c(t) |N(X)| \cdot |X'|^{1-1/t}}{ (|N(X)|-s)/(k-d)}.$$ Solving the inequality for $|X'|$, we have that 
 $$|X'| \leq \left(\frac{c(t) |N(X)| (k-d)}{ (|N(X)|-s)}\right)^t \leq c(t, k, d), $$
 for a constant $c(t,k,d)$, when $s<|N(X)|/2$. Since $|N(X)|\geq q$, the above inequality holds in particular when $s<q/2$.\\
 
 Finally, we shall consider $x$, the total number of $k$-element subsets of $Q$. On the one hand, $x = \binom{|Q|}{k} = \binom{q}{k}$. Since each $k$-element subset $V$ of $Q$ contains a subset $X$ of size $d$ and a vertex $y$ that is bad for $X$, 
 $x\leq \binom{q}{d} \cdot c(t,k, d) \cdot \binom{q-d-1}{k-d-1}= O(q^{k-1})$, where the first term is the number of ways to choose $X$ from $Q$, the second one is the number of ways to choose $y$, and the third term is the number of ways to choose the remaining $k-d-1$ elements of $V$. 
 So, we have that $x = o(\binom{q}{k})$, that contradicts the fact that $x = \binom{q}{k}$. This proves Claim 1.\\
 
 \noindent
 Assume from now on that there are no red $q$-cliques in $\cH$.\\

 {\bf Claim 2.~} If $Q$ is a $q$-clique in $\cH$ and for any distinct subsets $X, X'\subseteq Q$ of size $d$ each one has $N(X)\cap N(X') = \emptyset$, then $H(k, d, s)$ is an induced subgraph of $G$.\\
 
 \noindent
 From the definition of $\cH$ it follows that for each $X\subseteq Q$, $|X|=d$, we have $|N(X)|\geq s$.
 Let $B' = \cup_{X\in \binom{Q}{d}} N(X)$. 
 Since $q\geq k$, this immediately implies that $H(k, d, s)$ is a subgraph of $G[Q\cup B']$.
 To show that $H(k, d, s)$ is an induced subgraph of $G[Q\cup B']$, assume otherwise, i.e., that there is a vertex $b\in B'$ such that $b\in N(X)$ and $b\sim a$, for some $a\in Q\setminus X$. Then we see that $b\in N(X) \cap N(X-\{x\}\cup \{a\})$, for any $x\in X$. This contradicts the assumption of the Claim and proves it. \\
 
 {\bf Claim 3.~} The number of blue $q$-cliques in $\cH$ is at most $\kappa \cdot c \binom{|A|}{q}$, for a constant $c=c(q, d)$.\\
 
 By Claim 2, we can assume that for each blue $q$-clique $Q$ there is a vertex $b=b(Q)\in B$ such that $b\in N(X')$, where $X'\subseteq Q$, $|X'|=d+1$. 
 Let us denote some vertex from $X'$ as $a(Q)$. 
 To count blue $q$-cliques, we shall first upper bound the number of subsets of $A$ of size $q-1$ corresponding to $Q-\{a(Q)\}$ and then, for each such a fixed subset, we upper bound the number of ways to extend it to a $q$-set with a vertex corresponding to $a(Q)$. 
 Specifically, we have that the number of $q$-cliques in $\cH$ is at most 
 $$\binom{|A|}{q-1} \cdot \binom{q-1}{d}(q-1) \cdot \kappa |A|, $$
 where the second term $ \binom{q-1}{d}(q-1)$ corresponds to the largest number of vertices in $B_{Q'}= \cup_{X\in \binom{Q'}{d}} N(X)$, for any $Q' \in \binom{A}{q-1}$; 
 and the last term is the largest number of ways to find a neighbour $a$ of a vertex $b\in B_{Q'}$ (here we use the condition on the maximum degree of vertices in $B$ stated in the Proposition).
 For $c(q,d) = \binom{q-1}{d}(q-1)$, Claim 3 follows.\\ 
 
 {\bf Claim 4.~} The number of blue $q$-cliques in $\cH$ is at least $C\binom{|A|}{q}$, where $C=C(q, d, s, \eta)$.\\
 
 \noindent
 Recall that $R $ is the Ramsey number $R_d(q)$ and it is implicit that $R\geq 2q$.
 Consider a set $S\subseteq B$ such that $|S|=s$ and consider $N(S)$. Then we see that $\cH$ restricted to $N(S)$, denoted $\cH[N(S)]$, is a clique. We have that $\cH[N(S)]$ contains at least 
 $ (|N(S)|-2R)/q $ pairwise disjoint monochromatic $q$-cliques. Indeed, iteratively find such a clique as long as the number of vertices is greater than $R$, delete its vertices and proceed. 
 Since there are no red $q$-cliques, we have that the selected $q$-cliques are blue. 
 Let $C(S)$ be a maximal set of pairwise disjoint blue $q$-cliques in $\cH[N(S)]$.
 Let $\cC = \cup_{S\in \binom{B}{s}} C(S)$. We shall argue that the number of blue $q$-cliques is large using the Hypergraph Removal Lemma applied to the $d$-uniform hypergraph $\cH_{\rm blue}$ whose edges are blue edges of $\cH$.
 
 Since the members of $\cC$ are blue, for any blue $q$-clique $Q$, $|N(Q)|\leq q-1$ and in particular there are at most $\binom{q-1}{s}$ ways to choose a set $S$ of size $s$ from $N(Q)$. Thus 
 \begin{equation}\label{C}
 |\cC| \geq \frac{1}{\binom{q-1}{s}} \sum_{S\in \binom{B}{s}} \left(\frac{|N(S)|}{q} - \frac{2R}{q} \right).
 \end{equation}

 Since
 $ \sum_{S\in \binom{B}{s}} |N(S)|$ counts all stars in $G$ with exactly $s$ leaves in $B$, we have that 
 \begin{eqnarray*}
 \sum_{S\in \binom{B}{s}} |N(S)| &= & \sum_{a\in A} \binom{deg(a)}{s} \\
 &\geq & |A| \binom{||G||/|A|}{s}\\
 &\geq & \frac{ 1 }{s ^ s } \frac{ ||G|| ^ s }{ |A| ^ { s-1 } } \\
 &\geq & \frac{ 1 }{ s ^ s } \frac{ \left( \xi \eta n ^ { \frac{ d + s - 1 }{ d } } \right) ^ { s } }{ \left( \xi n ^ { \frac{ s }{ d } } \right) ^ { s-1 } }\\
 &=& \xi \frac{ \eta ^ s }{ s ^ s } n^s.
 \end{eqnarray*}
 
 We used Jensen's inequality in the second line and the bounds $||G|| \geq \xi \eta n ^ { \frac { s }{ d } }$, $|A| \leq \xi n^{s/d} = \xi n ^ { \frac{ d + s - 1 }{ d } }$ given in the statement of the Proposition. Coming back to (\ref{C}) and recalling that $|B|=n$, we have
$$|\cC| \geq \frac{1}{\binom{q-1}{s}}\frac{1}{q}\left( \xi \frac{ \eta ^ s }{ s ^ s } n^s - 2Rn^s\right).$$
This implies that $|\cC| \geq z_1 n ^ s \geq z \binom{|A|}{d} $, 
where $z_1= z_1(q, d, s, \eta, \xi)$
 and $z = z(z_1, \xi)$, $z \geq \xi^{-d} z_1>0$, for $\xi> 2Rs^s/ \eta^s$. \\

 Next we estimate the largest number of members of $\cC$ that contain a fixed $d$-element subset $A'$ of $A$, $A'\in \cE$.
 In case some member of $\cC$ contains $A'$, we see that $A'$ is blue. That implies that that $|N(A')|\leq q-1$, so there are at most $\binom{q-1}{s}$ ways to choose an $s$-element subset $S$ in $N(A')$. 
 Each such subset $S$ contributes to at most one members of $\cC$ from $C(S)$ that contains $A'$. 
Deleting an edge from $\cE$ destroys at most $\binom{q-1}{s}$ members of $\cC$. Thus one needs to delete at least $|\cC|/\binom{q-1}{s} \geq \epsilon \binom{|A|}{d}$ edges from $\cH_{\blue}$ to destroy all members of $\cC$. Here, 
$\epsilon = \epsilon(q, d, s, \eta, \xi) \geq z/\binom{q-1}{s}$. Thus by the Hypergraph Removal Lemma with $C = C(q,d, s, \eta, \xi) =\delta(\epsilon, d, q)$, we have that the number of $q$-cliques in $\cH_{\rm blue}$ is at least $C\binom{|A|}{q}$.
 This proves Claim 4. \\
 
 Now, Claim 3 and Claim 4 give a contradiction by choosing $\kappa$ sufficiently small, specifically by choosing $\kappa \ll C(q, d, s, \eta)/c(q, d)$. This concludes the proof of the proposition.
\end{proof}

\vskip 0.5cm

\begin{proof}[Proof of Theorem \ref{thm:main}]
 Let $G'$ be a $K_{t,t}$-free bipartite graph with parts of size $n$, for $n$ sufficiently large, such that $||G'||\geq \epsilon n^{2-1/d}$, for some positive $\epsilon$. We shall show that
 $G'$ contains $W(k, d, r)$ as an induced subgraph with body in a specified part of $G'$. \\

 Let  $G = (A', B'; E')$ be an induced subgraph of $G'$ guaranteed by the Reduction Lemma \ref{lem:reduction}. Specifically, 
 $|G|$ is sufficiently large,  $G$ is $K$-almost regular for some constant $K= K(\epsilon)  > 0$, and 
 \begin{equation} \label{eq:lbG}
 ||G|| \geq (2\epsilon/5) |G|^{2-1/d} \geq (2\epsilon/5) |B'|^{2-1/d}.
 \end{equation}
Since $G$ is $K$-almost regular,  $|A'|\delta(G) \leq ||G|| \leq |B'| \Delta(G)$, we have that $|A'| \leq |B'|\Delta(G)/\delta(G) \leq K|B'|$.  In addition, since $||G||\leq |B'| \Delta(G) \leq |B'|\delta(G) K$,  we have that 
\begin{equation}\label{eq:delta}
\delta(G) \geq \epsilon' |B'|^{1-1/d},
\end{equation} 
 for $\epsilon' = 2\epsilon / (5K)$. Let  $\eta =  \epsilon'/4$ and $s=d-r-1$. Consider the constants $\kappa$ and $\xi$ given by the Proposition \ref{prop}.\\

 The idea of the proof is as follows. We shall find a subset $X$ of $B'$ of size $r$ and a subset $A$ of $A'$ such that $X\sim A$ (when $r\neq 0$) and  a subset $B$ of $B'-X$ such that $G[A,B]$ satisfies conditions for Proposition \ref{prop}. As a result, there is an induced copy of $H(k, d, d-r-1)$ in $G[A, B]$ with body in $A$. Together with $X$, it forms an induced  copy of $W(k, d, r)$ with its body in $A$. \\
 
{\bf Claim 1. }For $r>0$,  there are sets $X\subseteq B'$ and $A\subseteq A'$ such that  $$|X|=r,  ~|A|=\frac{\xi}{2} |B'|^{1 - r/d-1/d}, ~ \mbox{  and  } X\sim A.$$\\
To prove the claim,  apply Lemma \ref{thm:KST}  with $y_1=r$,  $y_2=\frac{\xi}{2} |B'|^{1 - r/d-1/d}$, $|Y_1|= |B'|$ and $|Y_2|\leq K|B'|$,  to get that 
\begin{eqnarray*}
||G|| &\leq & (y_2-1)^{\frac{1}{y_1}}(|Y_1|-y_1+1)|Y_2|^{1-\frac{1}{y_1}}+(y_1-1)|Y_2|  \\
& \leq & 2y_2^{\frac{1}{y_1}}|Y_1||Y_2|^{1-\frac{1}{y_1}}\\
&\leq &
2 (\xi |B'|^{1-\frac{r}{d} - \frac{1}{d}})^{\frac{1}{r}} |B'| (K|B'|)^{1-\frac{1}{r}} \\
&= &2 \xi^{\frac{1}{r}} K^{1- \frac{1}{r}} |B'|^{2-\frac{1}{d}- \frac{1}{rd}}< \frac{2\epsilon}{5} |B'|^{2-\frac{1}{d}},
\end{eqnarray*}
a contradiction to the lower bound on $||G||$, see (\ref{eq:lbG}). Note that the last inequality holds since $|B'|$ is sufficiently large. This proves Claim 1. \\


Let $0\leq \gamma< 1/t$, in particular $\gamma< 1-1/d$.  Let $V_{ \text{bad} } = \{ b \in B': |N(b) \cap A| \geq |A| ^ { 1 -\gamma}\}.$\\

{\bf Claim 2.}  $|V_{ \text{bad} }| \leq c(t) |A|^{\gamma}.$\\

From the definition of $V_ {\text{bad}}$, $||A, V_{ \text{bad} } || \geq |V_{ \text{bad}}| |A| ^ {1 -\gamma }$.
 Since $G$ is $K_{t,t}$-free, Lemma \ref{thm:KST} gives that $||A, V_{ \text{bad} }|| \leq (C/2) \left( |A| ^ {\frac{ t-1 }{t}} |V_{ \text{bad} }| + |A| \right)$, for $C=C(t)>0$.
Comparing the lower and upper bound on $||A, V_{ \text{bad} } ||$ gives 
$|V_{ \text{bad} }| \leq \frac{C|A| } { |A| ^ {1 - \gamma} - C|A| ^ {1-1/t}} \leq c|A|^{\gamma}, $ for $c=2C$.
This proves Claim 2.\\

{\bf Claim 3.}  Let $B = B'-V_{\text{bad}}$. Then $G[ A,B ]$ satisfies the conditions of Proposition \ref{prop} with $s=d-r-1$.\\
 
Note that $|A|= (\xi/2) |B'|^{1-r/d-1/d} = (\xi/2) |B'|^{s/d}$.
\begin{itemize}
\item{}
We have that $|B| \geq |B'|/2$, so $|A| = (\xi/2) |B'|^{s/d} \leq (\xi/2) (2|B|)^{s/d} \leq \xi |B|^{s/d}.$

\item{}
Since $|B|$ is large enough, the definition of $V_{\text {bad}}$ implies that  
$\text{max}_{ b \in B }  \deg_{G}(b)  \leq |A|^{1- \gamma}  \leq \kappa |A|.$

\item{}
Using (\ref{eq:delta}),  we have $\delta(G) \geq \epsilon' |B'|^{1-1/d}\geq  \epsilon' |B|^{1-1/d}$. Since $|A| = (\xi/2) |B'|^{s/d}\geq  (\xi/2) |B|^{s/d}$ and $X\subseteq V_{\text{bad}}$,  we have 
 \begin{eqnarray*}
  || A, B || & \geq & |A| (\delta(G) - |V_{\text{bad}}|) \\
  &\geq &  \frac{\xi}{2} |B|^{\frac{s}{d}} \left( \epsilon' |B|^{1-\frac{1}{d}} - 2c|A|^{\gamma}\right) \\
 &\geq & \xi |B|^{\frac{s}{d}} \frac{\epsilon'}{4} |B|^{1-\frac{1}{d}} \\
 &=& \xi\eta |B|^{\frac{d+s-1}{d}}. 
 \end{eqnarray*}
 \end{itemize}
 This proves Claim 3. \\

 Thus, by Proposition \ref{prop}, $G[ A,B]$ contains a copy of a hedgehog $H=H(k,d,s)$ with body in $A$. Recall that $X\subseteq V_{\text{bad}}$, so   $X\cap B= \emptyset$. Thus, the vertices of $H$ together with $X$ induce $W(k,d,r)$ with body in $A$ in $G$ and thus in the original graph $G'$.\\
 

 For the second part of the theorem, we use the Deletion-method Lemma \ref{lem:deletion}. Let $\delta >0$ be given. 
 Recall that $\gamma(H)= \frac{|H|-2}{||H||-1}$. We have that $\gamma(K_{t,t}) = (2t-2)/(t^2-1) = 2/(t+1)\leq 1/d + \delta$ for $t\geq 2d-1$ and 
 $$\gamma(W(k,d,r)) = \frac{k+ r+ (d-r-1) \binom{k}{d}-2}{rk + (d-r-1)d\binom{k}{d} -1} <\frac{1}{d}+\delta,$$ for large enough $k$. Thus by Lemma \ref{lem:deletion}, there is a bipartite graph $G$ with $n$ parts in each side, such that $K_{t,t}, W(k,d,r) \not\subseteq G$ and $||G||\geq n^{2-\frac{1}{d}-\delta}$. \end{proof}

\begin{proof}[Proof of Theorem \ref{thm:ST}]
 Let $d \geq 2$ be an integer and $H=(V', U'; E)$ be a $K_{d,d}$-free bipartite graph $H$ 
 such that every vertex in part $U'$ of $H$ has degree at most $d$. By Lemma \ref{lem:propertiesW} (1) we have that 
 $H\subseteq W(k, d, 0)$.
 Let $G$ be a graph on $n$ vertices and $\epsilon n^{2-1/d}$ edges, for some positive $\epsilon$. By considering the max cut of $G$ we may assume $G$ to be bipartite.
 Let $t=\max \{|V'|, |U'|\}$.
 If $G$ contains a copy of $K_{t,t}$, it contains a copy of $H$ as a subgraph.
 Otherwise, by Theorem \ref{thm:main}, $G$ contains a copy of $W(k, d, 0)$, that in turn contains a copy of $H$.
\end{proof}

\begin{proof}[Proof of Theorem \ref{thm:JP}]
 Let $k$ and $d$ be integers such that $t\geq d\geq 3$.
 Let $G$ be a bipartite graph on partite sets $A,B$ of size $n/2$ each and having VC-dimension at most $d$ with respect to $A$.
 Assume that $K_{t,t}\nsubseteq G$ and $||G|| \geq \epsilon n^{2-\frac{1}{d}}$ for some positive constant $\epsilon$.
 Then Theorem \ref{thm:main} yields that $G$ contains an induced copy of $W(d+1, d, d-2)$ with body in $A$, that, by Lemma \ref{lem:propertiesW} (2),
 contains an induced copy of the incidence graph of a hypergraph with vertex set $[d+1]$ and edge set $2^{[d+1]}$. 
 However, by the definition of VC-dimension, this implies that the VC-dimension of $G$ is at least $d+1$, a contradiction.
\end{proof}

\section{Concluding remarks}

In this paper, we provide an upper bound on the largest number of edges in a bipartite graph with $n$ vertices in each part,  not containing a given bipartite graph $H$ as as induced subgraph and not containing a large but fixed biclique, for any $H$ that is $K_{d,d}$-free and such that any vertex in one part has degree either at most $d$ or a full degree, $d\geq 2$.  It remains interesting to determine the induced Tur\'an function of an arbitrary bipartite graph and a bi-clique.
In addition, it would be very nice to resolve a conjecture of Conlon and Lee, \cite{CL} that  states that for any $K_{d,d}$-free bipartite graph $H$ with maximum degree at most $d$ in one part, there are positive constants $C$ and $\delta$ such that $\ex(n, H) \leq Cn^{2-1/d - \delta}$.

Lemma \ref{lem:ind} shows that in most cases  $\ex(n, \{F, H \ind\})$ is equal to either $\ex(n, F)(1+o(1))$ or $\ex(n, H)(1+o(1))$.  
There are pairs of graphs $F$ and $H$ for which it doesn't hold.  For example, consider $F=K_{t,t}$ and $H=P_\ell$, the path on $\ell$ edges, $\ell>1$.
Then $(t-1)n +o(n) \leq \ex(n, \{K_{t,t}, P_\ell\ind\}\}) \leq t^c n  $, where the lower bound  can be seen by considering pairwise vertex-disjoint union of cliques on $2t-1$ vertices each and the upper bound follows from a result by Scott, Seymour, and Spirkl \cite{SSS} mentioned in the introduction. 
On the other hand, $\ex(n, K_{t, t}) \geq cn^{2 - 2/{t+1}}$ and $\ex(n, P_\ell) \leq \ell n$. So, when $t$ is much larger than $\ell$, we see that  
$$ex(n, \{K_{t,t}, P_\ell\ind\}\}) \not\in \{ \ex(n, P_\ell)(1+o(1)), \ex(n, K_{t, t})(1+o(1))\}.$$ 
However, it is not clear whether 
$$\ex(n, \{F, H \ind\})\in \{ \Theta(\ex(n, F)), \Theta(\ex(n, H))\}$$   for all graphs $F$ and $H$. 
In the manuscript by Hunter, Milojevi\'c, Sudakov, and Tomon \cite{HMST} the authors independently state a conjecture similar to our question:
$$  \ex(n, \{K_{t,t}, H\ind\})\leq C(H, t) \ex(n, H),$$ for a constant $C(H,t)$ depending only on $H$ and $t$.\\

Note that instead of  $\ex(K_n, \{K_{t,t}, H\ind\})$, one can consider  $\ex(K_n, \{K_{t,t}, \cH\ind\})$ for a class of graphs $\cH$. When $\cH$ is the set of all subdivisions of $H$, this problem attracted a lot of attention, starting with a work of  K\"uhn and Osthus \cite{KO}.   Recently, the problem has been considered by  Gir\~ao and Hunter \cite{GH},   McCarty \cite{M}, Bourneuf,  Buci\'c, Cook, and Davies \cite{BBCD}, as well as by Du, Gir\~ao, Hunter, McCarty,  and Scott  \cite{DGHMS}.  \\

It is also not clear in general  whether  $\ex(K_n, \{F, H \ind\}) = O(\ex(K_{n,n}, \{F, H \ind\}))$. In the non-induced setting it holds by considering a maximal balanced cut.
We would like to conclude with a question whether our main result holds for $K_n$  instead of $K_{n,n}$:   
Is it true that for any non-negative integers $k,r,d,t$ with $d \geq r+2$, we have $\ex(K_n, \{K_{t,t}, W(k,d,r) \ind\}) = o(n ^ {2 - \frac{ 1 }{ d }})$?\\

\vskip 1cm
\noindent
{\bf Acknowledgments} The research of the first author was partially supported by the DFG grant FKZ AX 93/2-1. The authors thank Matija Buci\'c for bringing their attention to Lemma 7.1. in \cite{GH} and for interesting discussions.

\end{document}